\documentclass[a4paper,10pt]{amsart}
\usepackage{color}

\usepackage[english]{babel}
\usepackage{amssymb}
\usepackage[sort]{natbib}
\usepackage{mathrsfs}
\usepackage{comment}
\usepackage{mathscinet}
\usepackage{fullpage}
\usepackage[colorlinks,linkcolor=blue,citecolor=blue]{hyperref}
\usepackage{kamil}

\begin{document}
\bibliographystyle{plainnat}
\setcitestyle{numbers}

\title{An Erd\"os--R\'ev\'esz type law of the iterated logarithm for
order statistics of a stationary Gaussian process}

\author{K.\ D\polhk{e}bicki}
\address{Instytut Matematyczny, University of
Wroc\l aw, pl.\ Grunwaldzki 2/4, 50-384 Wroc\l aw, Poland.}
\email{Krzysztof.Debicki@math.uni.wroc.pl}

\author{K.M.\ Kosi\'nski}
\address{Instytut Matematyczny, University of
Wroc\l aw, pl.\ Grunwaldzki 2/4, 50-384 Wroc\l aw, Poland.}
\email{Kamil.Kosinski@math.uni.wroc.pl}



\date{\today}

\begin{abstract}
Let $\{X(t):t\in\mathbb R_+\}$ be a stationary Gaussian process with almost surely (a.s.)
continuous sample paths, $\mathbb E X(t) = 0$, $\mathbb E X^2(t) = 1$ and correlation function satisfying
(i) $r(t)  = 1 - C|t|^{\alpha} + o(|t|^{\alpha})$ as $t\to 0$ for some $0\le\alpha\le 2, C>0$; 
(ii) $\sup_{t\ge s}|r(t)|<1$ for each $s>0$ and
(iii) $r(t) = O(t^{-\lambda})$ as $t\to\infty$ for some $\lambda>0$.
For any $n\ge 1$, consider $n$ mutually independent copies of $X$ and denote by
$\{X_{r:n}(t):t\ge 0\}$ the $r$th smallest order statistics process, $1\le r\le n$.
We provide a tractable criterion for assessing whether, for any positive, non-decreasing function $f$, 
$\mathbb P(\mathscr E_f)=\mathbb P(X_{r:n}(t) > f(t)\, \text{ i.o.})$  equals 0 or 1.
Using this criterion we find that, for a family of functions 
$f_p(t)$, such that $z_p(t)=\mathbb P(\sup_{s\in[0,1]}X_{r:n}(s)>f_p(t))=\mathscr C(t\log^{1-p} t)^{-1}$, $\mathscr C>0$, $\mathbb P(\mathscr E_{f_p})= 1_{\{p\ge 0\}}$. 
Consequently, with $\xi_p (t) = \sup\{s:0\le s\le t, X_{r:n}(s)\ge f_p(s)\}$, for $p\ge 0$, 
$\lim_{t\to\infty}\xi_p(t)=\infty$ and $\limsup_{t\to\infty}(\xi_p(t)-t)=0$ a.s.
Complementary, we prove an Erd\"os--R\'ev\'esz type law of the iterated logarithm lower bound on $\xi_p(t)$, i.e., 
$\liminf_{t\to\infty}(\xi_p(t)-t)/h_p(t) = -1$ a.s., $p>1$; $\liminf_{t\to\infty}\log(\xi_p(t)/t)/(h_p(t)/t) = -1$ a.s., $p\in(0,1]$, where $h_p(t)=(1/z_p(t))p\log\log t$.
\end{abstract}

\maketitle

\section{Introduction and Main Results}
\label{sec:intro}
Let $X=\{X(t):t\in\rr_+\}$ be a stationary Gaussian process with almost surely (a.s.)
continuous sample paths, $\ee X(t) = 0$ and $\ee X^2(t) = 1$.
Suppose that the correlation function of $X$, $r(t) = \ee X(t) X(0)$,
satisfies the following regularity assumptions:
\begin{align}
\label{eq:rasympt}
r(t) & = 1 - C|t|^{\alpha} + o(|t|^{\alpha})\quad \text{as } t\to 0\text{ for some }0\le\alpha\le 2,\quad C>0,\\
r^*(s)&=\sup_{t\ge s}|r(t)|  <1\quad\text{for each } s>0,\nonumber\\
\label{eq:r*}
r(t) &= O(t^{-2\lambda})\quad\text{as } t\toi\text{ for some }\lambda>0.
\end{align}
The analysis of extremes of Gaussian stochastic processes has a long history.
The celebrated \textit{double sum} method, primarily developed by Pickands, e.g., \cite{Pickands69a},
and extended by seminal works of Piterbarg, e.g., \cite{Piterbarg78} or monograph \cite{Piterbarg96},
plays central role in the extreme value theory of Gaussian processes. The
technique developed there appeared to be an universal method, which may deliver answers
also to classes of non-Gaussian processes, see for example, recent contributions of
\cite{Hashorva13,HKP15}.

Laws of the iterated logarithm take important place in
this theory, providing properties of extremal behavior of stochastic processes
on large-time scale. One of important contributions in this domain is
a result on the process $\xi=\{\xi(t):t\ge0\}$, defined via $\xi(t) = \sup\{s:0\le s\le t, X(s) \ge (2\log s)^{1/2}\}$.
In particular, the law of the iterated logarithm implies that, see \cite{Shao92, Qualls71},
\[
\limsup_{t\toi}(\xi(t)-t) = 0\quad \text{a.s.}
\]
Interestingly, under the above regularity assumptions, \citet{Shao92} gave the lower bound of $\xi(t)$ and
obtained an Erd\"os--R\'ev\'esz type law of the iterated logarithm, that is,
\begin{align}
\label{eq:Shao}
\liminf_{t\toi}\frac{\xi(t)-t}{t(\log t)^{(\alpha-2)/(2\alpha)}\cdot\log_2 t}
&=
-\frac{(2+\alpha)\sqrt{\pi}}{\alpha \mathcal H_{\alpha}(2C)^{1/\alpha}}\quad \text{a.s. if }
0<\alpha<2,\\
\label{eq:Shao2}
\liminf_{t\toi}\frac{\log\left(\xi(t)/t\right)}{\log_2 t}
&=
-\frac{2\sqrt{\pi}}{\mathcal H_{2}\sqrt{2C}}\quad \text{a.s. if }
\alpha=2,
\end{align}
where
$\mathcal H_{\alpha}$ is the {\it Pickands' constant}
defined by
$\mathcal H_{\alpha}=\lim_{T\toi} T^{-1} \mathbb{E}e^{\sup_{t\in[0,T]}( \sqrt{2}B_{\alpha/2}(t)-t^{\alpha})}$,
with
$B_{\alpha/2}=\{B_{\alpha/2}(t):t\ge 0\}$ denoting
fractional Brownian motion with Hurst index $\alpha/2\in(0,1]$, i.e.,
a continuous, centered Gaussian process with covariance function
\[
\ee B_{\alpha/2}(s)B_{\alpha/2}(t) =\half(|s|^\alpha+|t|^\alpha-|t-s|^\alpha).
\]
Equation \eqref{eq:Shao} shows that for any $t$ big enough there exists an $s$ in $[t-t(\log t)^{(\alpha-2)/(2\alpha)}\cdot\log_2 t, t]$ such that, almost surely, $X(s)\ge (2\log s)^{1/2}$  and that the length of the interval $t(\log t)^{(\alpha-2)/(2\alpha)}\cdot\log_2 t$ is smallest possible. Moreover, the bigger the parameter $\alpha$ is, the wider the interval will be.

In this paper, we derive a counterpart of Shao's result for the order statistics process $X_{r:n}$.
Namely, for any $n\ge 0$, we consider $X_1,\ldots,X_n$, $n$ mutually independent copies of $X$
and denote by $X_{r:n}=\{X_{r:n}(t):t\ge 0\}$ the $r$th smallest order statistics process, that is,
for each $t\ge 0$, $1\le r\le n$,
\[
X_{1:n}(t) = \min_{1\le j\le n} X_j(t)\le X_{2:n}(t)\le\ldots \le X_{n-1:n}(t)\le \max_{1\le j\le n} X_j(t) = X_{n:n}(t).
\]

Our first contribution is the theorem that extends classical findings of Qualls and Watanabe \cite{Qualls71}.
\begin{theorem}
\label{thm:equiv}
For all functions $f$ that are positive and non-decreasing on some interval $[T,\infty)$, $T>0$, it follows that
\[
\prob{\mathscr E_f}:=\prob{ X_{r:n}(t) > f(t)\quad\text{i.o.}} = 0 \quad \text{or}\quad 1,
\]
as the integral
\[
\mathscr I_f := \int_T^\infty \prob{\sup_{t\in[0,1]} X_{r:n}(t) > f(u)}\D u\quad \text{is finite or infinite}.
\]
\end{theorem}
\noindent \citet[Theorem 2.2]{Debicki14}, see also \cite{DHJT15}, gave the expression for the asymptotic behavior of the probability in $\mathscr I_f$, namely
\begin{equation}
\label{eq:asymptotics}
\prob{ \sup_{t\in[0,1]} X_{r:n}(t) > u}
= C^{\frac{1}{\alpha}}\binom{n}{\hat r}\mathcal H_{\alpha,\hat r}  u^{\frac{2}{\alpha}} \left(\Psi(u)\right)^{\hat r}(1+o(1)),\quad \as u,
\end{equation}
where $\hat r = n-r+1$, $\Psi(u)=1-\Phi(u)$ and $\Phi(u)$ is the distribution function of unit normal law,
\[
\mathcal H_{\alpha, k} = \lim_{T\toi} T^{-1}\mathcal H_{\alpha,k}(T)\in(0,\infty),
\]
\[
\mathcal H_{\alpha,k}(T)=\int_{\rr^k} e^{\sum_{i=1}^k w_i}
\prob{\sup_{t\in[0,T]}\min_{1\le i\le k}
\left(\sqrt 2 B_{\alpha/2}^{(i)}(t) - t^{\alpha} - w_i\right) > 0
}
\D w_1\ldots\D w_k
\]
and
$B_{\alpha/2}^{(i)}$, $1\le i\le n$,
are mutually independent fractional Brownian motions.
$\mathcal H_{\alpha, k}$ is the \textit{generalized Pickands' constant}
introduced in \citep{Debicki15}; see also \citep{Debicki14}.
Therefore, \autoref{thm:equiv} provides a tractable criterion for settling the dichotomy of $\prob{\mathscr E_f}$.

For instance, let
\[
f_p(s)  = \left(\frac{2}{\hat r}\left(\log s + \left(\frac{2-\hat r\alpha}{2\alpha} +1 -p\right)\log_2 s\right)\right)^{\frac{1}{2}},\quad p\in\rr.
\]
One easily checks that, as $u\toi$,
\begin{equation}
\label{eq:Gf}
\prob{ \sup_{t\in[0,1]} X_{r:n}(t) > f_p(u)} = C^{\frac{1}{\alpha}}\binom{n}{\hat r}\frac{\mathcal H_{\alpha,\hat r}}{(2 \pi)^{\frac{\hat r}{2}}}
\left(\frac{2}{\hat r}\right)^{\frac{2-\hat r\alpha}{2\alpha}}
\left(u\log^{1-p} u\right)^{-1}(1+o(1)).
\end{equation}
Hence, for any $p\in\rr$,
\[
\prob{X_{r:n}(t) > f_p (t)\quad\text{i.o.}}=
\left\{
\begin{array}{cc}
1 & \text{if } p\ge 0 \\
0 & \text{if } p<0  \\
\end{array}
\right. .
\]
Furthermore,
\[
\limsup_{t\toi} \frac{X_{r:n}(t)}{\sqrt{\log t}} =\sqrt\frac{2}{\hat r}
\quad\text{a.s.}
\]
\\
Next, consider the process $\xi_p=\{\xi_p(t):t\ge 0\}$ defined as
\[
\xi_p(t)=\sup\{s:0\le s\le t, X_{r:n}(s)\ge f_p(s)\}.
\]
Since $\mathscr I_{f_p} = \infty$ for $p\ge 0$, \autoref{thm:equiv} implies that
\[
\lim_{t\toi} \xi_p(t) = \infty\quad\text{a.s.} \quad\text{and}\quad
\limsup_{t\toi}(\xi_p(t) - t) = 0\quad\text{a.s.}
\]
Let, cf. \eqref{eq:Gf},
\[
h_p(t)= p\left(\prob{ \sup_{s\in[0,1]} X_{r:n}(s) > f_p(t)} \right)^{-1}\log_2 t.
\]
The second contribution of this paper is an Erd\"os--R\'ev\'esz type of law of the iterated logarithm for the process $\xi_p$.
\begin{theorem}
\label{thm:main}
If $p>1$, then
\[
\liminf_{t\toi}\frac{\xi_p(t)-t}{h_p(t)} = - 1\ \ {\rm a.s.}
\]
If $p\in(0,1]$, then
\[
\liminf_{t\toi}\frac{\log\left(\xi_p(t)/t\right)}{h_p(t)/t } = - 1\ \ {\rm a.s.}
\]
\end{theorem}

Now, let us complementary put $\eta_p = \{\eta_p(t):t\ge 0\}$, where
\[
\eta_p(t) = \inf\{s\ge t: X_{r:n}(s)\ge f_p(s)\}.
\]
Since
\[
\prob{\xi_p(t) - t\le - x} = \prob{\sup_{s\in(t-x,t]}\frac{X_{r:n}(s)}{f_p(s)}< 1}\]
and
\[
\prob{z - \eta_p(z)\le - x} = \prob{\sup_{s\in[z,z+x]}\frac{X_{r:n}(s)}{f_p(s)}< 1},
\]
then it follows that
\begin{equation}
\label{eq:eq}
\liminf_{t\toi}\frac{\xi_p(t)-t}{h_p(t)}=\liminf_{z\toi}\frac{z-\eta_p(z)}{h_p(z)}.
\end{equation}

\autoref{thm:main} shows that for $t$ big enough,
there exists an $s$ in $[t - h_p(t), t]$ (as well as in $[t, t+ h_p(t)]$ by \eqref{eq:eq}) such that
$X_{r:n}(s)\ge f_p(s)$ and that the length of the interval
$h_p(t)$ is smallest possible.
One can retrieve \eqref{eq:Shao}-\eqref{eq:Shao2} by setting $n=1$, and
$p= \frac{2-\hat r\alpha}{2\alpha} + 1 = \frac{2+\alpha}{2\alpha}$.
\autoref{thm:main} not only generalizes \citet[Theorem 1.1]{Shao92}, it also unveils the lacking so far structure of the lower bound of $\xi_p(t)$ by relating it, via $h_p(t)$, to the asymptotics of the tail distribution of the supremum of the underlying process evaluated at $f_p(t)$; in \eqref{eq:Shao} $t(\log t)^{(\alpha-2)/(2\alpha)}$ is of the same asymptotic order as the reciprocal of $\prob{\sup_{s\in[0,1]} X(s) > (2\log t)^{1/2}}$.
This shines new light on this type of results, which appear to be intrinsically connected with Gumbel limit theorems; see, e.g., \citep{Leadbetter83}, where the function $h_p(t)$ plays crucial role. We shall pursue this elsewhere.

The paper is organized as follows.
In \autoref{sec:auxLem} we provide a collection of basic results  on order statistics of stationary Gaussian processes, used throughout the paper, and prove auxiliary lemmas, which constitute building blocks of the proofs of the main results. These are given in the final part of the paper, \autoref{sec:Proofs}.

\section{Auxiliary Lemmas}
\label{sec:auxLem}
We begin with some auxiliary lemmas that are later needed in the proofs.
\\
The following lemma is the general form of the Borel--Cantelli lemma; cf. \citep{Spitzer64}.
\begin{lem}
Consider a sequence of events $\{E_k:k\ge0\}$.
 If
\[
\sum_{k=0}^\infty \prob{E_k} < \infty,
\]
then $\prob{E_n\text{ i.o.}} = 0$. Whereas, if
\[
\sum_{k=0}^\infty \prob{E_k} = \infty
\quad\text{and}\quad
\liminf_{n\toi}\frac{\sum_{1\le k\ne t\le n}\prob{E_k E_t}}{\left(\sum_{k=1}^n\prob{E_k}\right)^2}\le 1,
\]
then $\prob{E_n\text{ i.o.}} = 1$.
\end{lem}

The following two lemmas constitute useful tools for approximating the supremum of $X_{r:n}$ on a fixed interval by its maximum on a grid with a sufficiently dense mesh.
\begin{lem}
\label{lem:discreate_approx}
There exist positive constants $K,c$ and $u_0$ such that
\begin{align*}
\mathbb P \left(
\max_{0\le j\le u^{\frac{2}{\alpha}}/\theta }
X_{r:n}(j \theta u^{-\frac{2}{\alpha}})\le u - \frac{\theta^{\frac{\alpha}{4}}}{u},
\sup_{t\in[0,1]} X_{r:n}(t)>u\right)
&\le
K u^{\frac{2\hat r}{\alpha}}
\left(\Psi(u)\right)^{\hat r} \theta^{\frac{\alpha}{2}-1}\Psi(c \theta^{-\frac{\alpha}{4}}),
\end{align*}
for each $\theta>0$ and $u\ge u_0$.
\end{lem}
\begin{proof}
Note that, by stationarity, there exists a constant $K$, that may vary from line to line, such that, for sufficiently large $u$,
\begin{align*}
\mathbb P &\Bigg(\max_{0\le j\le u^{\frac{2}{\alpha}}/\theta }
X_{r:n}(j \theta u^{-\frac{2}{\alpha}})\le u - \frac{\theta^{\frac{\alpha}{4}}}{u},
\sup_{t\in[0,1]} X_{r:n}(t)>u\Bigg)\\
&\le
\frac{u^{\frac{2}{\alpha}}}{\theta}
\prob{X_{r:n}(0)\le u - \frac{\theta^{\frac{\alpha}{4}}}{u},\sup_{t\in[0,1]} X_{r:n}(t)>u}\\
&\le
\frac{u^{\frac{2}{\alpha}}}{\theta}
\binom{n}{r}\binom{n}{n-r+1}
\mathbb P\Big( \forall_{i=1,\ldots,r}\, X_i(0)\le u - \frac{\theta^{\frac{\alpha}{4}}}{u},
\forall_{j=r,\ldots,n}
\,
\sup_{t\in[0,1]} X_{j}(t)>u\Big)\\
&\le
K
\frac{u^{\frac{2}{\alpha}}}{\theta}
\prob{X_{r}(0)\le u - \frac{\theta^{\frac{\alpha}{4}}}{u}, \sup_{t\in[0,1]} X_{r}(t)>u}
\left(\prob{\sup_{t\in[0,1]} X(t)>u}\right)^{n-r}\\
&\le
K
 u^{\frac{2\hat r}{\alpha}}
\left(\Psi(u)\right)^{\hat r} \theta^{\frac{\alpha}{2}-1}\Psi(c \theta^{-\frac{\alpha}{4}}).
\end{align*}
The last inequality follows from \eqref{eq:asymptotics} and the classical result of \citet[Lemma 12.2.5]{Leadbetter83}, where the constant $c>0$ is given therein.
\end{proof}
The proof of the following lemma follows line-by-line the same reasoning as the proof of \citep[Theorem 2.2]{Debicki14} and thus we omit it.
\begin{lem}
\label{lem:disc_asymp}
For any $\theta>0$, as $u\toi$,
\[
\prob{ \max_{0\le j\le u^{\frac{2}{\alpha}}/\theta} X_{r:n}(j \theta u^{-\frac{2}{\alpha}}) > u}
=
C^{\frac{1}{\alpha}}\binom{n}{\hat r} \frac{\mathcal H_{\alpha,\hat r}(\theta)}{\theta}\left(\Psi(u)\right)^{\hat r}(1+o(1)).
\]
\end{lem}
The next lemma follows directly from \citep[Theorem 2.4]{DHJL15} and is a generalization of the classical Berman's inequality to order statistics.
\begin{lem}
\label{lem:Borrel}
For some $n,d\ge 1$, and any $1\le l\le n$ let
$\{\xi_l^{(0)}(i): 1\le i \le d\}$ and $\{\xi_l^{(1)}(i): 1\le i \le d\}$ be
a sequence of $\mathcal N(0,1)$ variables and
set $\sigma^{(\kappa)}_{il,jk} = E{\xi_l^{(\kappa)}(i)\xi_{k}^{(\kappa)}}(j)$, $\kappa=0,1$.
For any $1\le r\le n$ and $1\le i\le d$, let
$\xi_{r:n}^{(\kappa)}(i)$ be the $r$th order statistic
of $\xi_{1}^{(\kappa)}(i),...,\xi_{n}^{(\kappa)}(i)$.
Suppose that, for any $1\le i,j\le d, 1\le l,k\le n, \kappa=0,1$,
\[
\sigma_{il,jk}^{(\kappa)}=
\sigma_{ij}^{(\kappa)} 1_{\{l=k\}}
\]
for some $\sigma_{ij}^{(\kappa)}$. Now define
\[
 \rho_{ij} =\max\left(\left| \sigma_{ij}^{(0)}\right|, \left|\sigma_{ij}^{(1)}\right|\right),\quad
A_{ij}^{(r)} =
 \int_{\sigma_{ij}^{(0)}}^{\sigma_{ij}^{(1)}}
     \frac{\left(1+ |h|\right)^{(n-r)/2}}{(1-h^2)^{\hat r/2}}
     \, dh.
\]
Then, for any $u_1,\ldots,u_d>0$, for some positive constant $C_{n,r}$ depending only on $n$ and $r$,
\begin{align*}
&\prob{
\bigcap_{i=1}^d
\left\{\xi_{r:n}^{(0)}(i)\le u_i\right\}}
-
\prob{\bigcap_{i=1}^d\left\{\xi_{r:n}^{(1)}(i)\le u_i\right\}}
\\
\ &
\le
C_{n,r}
\sum_{1\le i<j\le d}
\left(u_i +u_j\right)^{-(n-r)}
\left(A_{ij}^{(r)}\right)^+
\exp\left(-\frac{\hat r\left(u_i^2+u_j^2\right)}{2(1+\rho_{ij})}\right).
\end{align*}
\end{lem}

\begin{lem}
\label{lem:bound}
Under the conditions of \autoref{thm:main}, for any $\varepsilon\in(0,1)$, there exist positive constants $K$ and $\rho$ depending only on $\varepsilon, \alpha$ and $\lambda$  such that
\[
\mathbb P \left(\sup_{S\le t\le T} \frac{X_{r:n}(t)}{f_p(t)}\le 1\right)
\le
\exp
\left(
	-\frac{(1-\varepsilon)}{(1+\varepsilon)}\int_{S+1}^T\prob{\sup_{t\in[0,1]}X_{r:n}(t)> f_p(u)}\D u
\right)
+ K S^{-\rho},
\]
for any $T-1\ge S\ge K$.
\end{lem}
\begin{proof}
Let, for any $i\ge 0$ and $\varepsilon\in(0,1)$,
\[
s_i = S + i(1+\varepsilon), \quad t_i = s_i + 1,\quad x_i = f_p(t_i), \quad I_i= (s_i, t_i].
\]
For some $\theta>0$, define grid points in the interval $I_i$, as follows
\begin{equation}
\label{def:sk}
s_{i,u} = s_i + uq_i,\quad 0\le u \le L_i,\quad L_i = [1/q_i], \quad q_i = \theta x_i^{-\frac{2}{\alpha}}.
\end{equation}
Since $f_p$ is an increasing function, it easily follows that, with $T(S,\varepsilon)=[(T-S-1)/(1+\varepsilon)]$,
\[
\prob{\sup_{S\le t\le T}\frac{X_{r:n}(t)}{f_p(t)}\le 1}
\le
\prob{\bigcap_{i=0}^{T(S,\varepsilon)}\left\{\sup_{t \in I_i} X_{r:n}(t) \le x_i\right\}}
\le
\prob{\bigcap_{i=0}^{T(S,\varepsilon)}\left\{\max_{0\le u \le L_i} X_{r:n}(s_{i,u}) \le x_i\right\}}.
\]

For any $1\le l\le n$ and $i\ge0$, let $X_{l,i}$ be an independent copy of the process $X_l$. Define a sequence of processes $Y_l=\{Y_l(t):t\in \cup_i I_i\}$ as
$Y_l(t) = X_{l,i}(t)$, if $t\in I_i$. Let $Y_{r:n}=\{Y_{r:n}(t):t\ge 0\}$ be the $r$th order statistic of $Y_1,\ldots, Y_n$.
Put
\begin{align*}
\sigma_{il,jk}^{(0)} &:= \ee X_{l}(i) X_{k}(j)=
r\left(|j-i|\right)1_{\{l=k\}} =: \sigma_{ij}^{(0)} 1_{\{l=k\}},\\
\sigma_{il,jk}^{(1)} &:= \ee Y_{l}(i) Y_{k}(j) =
r\left(|j-i|\right)1_{\{l=k\}} 1_{\{\exists m: i,j\in I_m\}} =:\sigma_{ij}^{(1)}
1_{\{l=k\}},
\end{align*}
and note that
\begin{align}
\nonumber
\rho_{ij} &=\max\left(\left| \sigma_{ij}^{(0)}\right|, \left|\sigma_{ij}^{(1)}\right|\right)=|r\left(|j-i|\right)|,\\
\label{def:A_il}
A_{ij}^{(r)} &= \int_{\sigma_{ij}^{(0)}}^{\sigma_{ij}^{(1)}} \frac{(1+|h|)^{2(n-r)}}{(1-h^2)^{\hat r/2}}\D h
= 1_{\{ \forall m: i,j\notin I_m\}}
\int_{0}^{r(j-i)} \frac{(1+|h|)^{2(n-r)}}{(1-h^2)^{\hat r/2}}\D h=:
1_{\{ \forall m: i,j\notin I_m\}} |\tilde A_{ij}^{(r)}|.
\end{align}
Now using \autoref{lem:Borrel} we find that
\begin{align*}
&\prob{\bigcap_{i=0}^{T(S,\varepsilon)}\left\{\max_{0\le u \le L_i} X_{r:n}(s_{i,u}) \le x_i\right\}}\\
&\le
\prod_{i=0}^{T(S,\varepsilon)}
\prob{\max_{0\le u \le L_i} X_{r:n}(s_{i,u}) \le x_i}\\
&\quad+
C_{n,r}\sum_{0\le i<j\le T(S,\varepsilon)}\sum_{\substack{0\le u\le L_i\\0\le  v \le L_j}}
 \left(x_{i} x_{j}\right)^{-(n-r)}
 \left|\tilde  A_{s_{i,u}s_{j,v}}^{(r)}\right|
 \exp\left(-\frac{\hat r\left(x_i^2+x_j^2\right)}{2(1+|r(s_{j,v}-s_{i,u})|)}\right)\\
&=:P_1 + P_2.
\end{align*}

\noindent\textit{Estimate of $P_1$.}

\vb

Since $X_{r:n}$ is a stationary process, from \autoref{eq:asymptotics}
combined with \autoref{lem:disc_asymp}, for any $\varepsilon\in(0,1)$, sufficiently large $\theta$ and $S$,
\begin{align*}
P_1 &\le \exp\left(-\sum_{i=0}^{T(S,\varepsilon)}
\prob{\max_{0\le u \le L_i} X_{r:n}(s_{i,u}) >x_i}
\right)
\le
\exp\left(-(1-\varepsilon)\sum_{i=0}^{T(S,\varepsilon)}
\prob{\sup_{t\in [0,1]} X_{r:n}(t) > f_p(t_i)}
\right)\\
&\le
\exp\left(-
\frac{1-\varepsilon}{1+\varepsilon}\int_{S+1}^T\prob{\sup_{t\in [0,1]} X_{r:n}(t) > f_p(u)}\D u
\right).
\end{align*}

\noindent\textit{Estimate of $P_2$.}

\vb

Noting that, for any $0\le i<j$, $0\le u\le L_i$, $0\le v\le L_j$;
\begin{align*}
s_{j,v}-s_{i,u}
=
s_j + v  q_j-s_i- u q_i
=
(j-i)(1+\varepsilon) +   v  q_j - uq_i
\ge
 (j-i)\varepsilon,
\end{align*}
we have
$$
\sup_{
	\substack{		
		0\le u\le L_i,\\
		0\le v\le L_j
	}
}
|r(s_{j,v}-s_{i,u})|
\le
\sup_{|s-s'|\ge (j-i)\varepsilon
}
|r(s-s')|=
r^*((j-i)\varepsilon)\le r^*(\varepsilon)<1.
$$
Without loss of generality assume that $\lambda < 2$. From \eqref{eq:r*} it follows that there is $s_0$ such that for every $s>s_0$,
\[
r^*(s)\le s^{-\lambda}\le \min(1,\lambda)/4.
\]
Finally, since the integrand in the definition of $\tilde  A_{s_{i,u}s_{j,v}}^{(r)}$ is continuous and bounded on $[0,r^*(\varepsilon)]$, there exists a generic constant $K$  not depending on $S$ and $T$, which may differ from line to line, such that
\[
\left|\tilde  A_{s_{i,u}s_{j,v}}^{(r)}\right|\le K |r(s_{j,v}-s_{i,u})|\le K r^*((j-i)\varepsilon).
\]
Therefore, for sufficiently large $S$,
\begin{align*}
P_2 &\le K
\sum_{0\le i<j\le T(S,\varepsilon)}
L_iL_j r^*\left((j-i)\varepsilon\right)
\exp\left(-\frac{\hat r(x_i^2+x_j^2)}{2(1+r^*\left((j-i)\varepsilon\right)}\right)\\
&\le
K\left(
\sum_{\substack{0<j-i\le 2s_0 \\ 0\le i<j\le T(S,\varepsilon)}} +
\sum_{\substack{j-i > 2s_0 \\ 0\le i<j\le T(S,\varepsilon)}}
\right) (\cdot)\\
&\le
K\Bigg(
\sum_{i=0}^\infty x_i^{\frac{4}{\alpha}}
\exp\left(-\frac{\hat r x_i^2}{1+r^*\left(\varepsilon\right)}\right)
+
\sum_{\substack{j-i > 2s_0 \\ 0\le i<j\le T(S,\varepsilon)}}
x_i^{\frac{2}{\alpha}}
x_j^{\frac{2}{\alpha}}
(j-i)^{-\lambda}
\exp\left(-\frac{\hat r(x_i^2+x_j^2)}{2(1+\frac{\lambda}{4})}\right)
\Bigg)\\
&\le
K\left(
\sum_{i=0}^\infty
t_i^{-\frac{2}{1+ \sqrt{r^*(\varepsilon)}}}
+
\sum_{\substack{j-i > 2s_0 \\ 0\le i<j\le T(S,\varepsilon)}}
 t_i^{-\frac{1}{1+\frac{\lambda}{2}}}
t_j^{-\frac{1}{1+\frac{\lambda}{2}}}
(j-i)^{-\lambda}
\right).
\end{align*}
We can bound the first sum from the above by
\[
K
\sum_{i=0}^{\infty}
	(S+i)^{-\frac{2}{1+\sqrt {r^*(\varepsilon)}}}
\le
K
S^{-\frac{1-\sqrt{r^*(\varepsilon)}}{4}}.	
\]
The second sum is bounded from above by
\begin{align*}
\sum_{S\le i<j<\infty}^\infty
&
i^{-\frac{1}{1+\frac{\lambda}{2}}}
j^{-\frac{1}{1+\frac{\lambda}{2}}}
(j-i)^{-\lambda}
=
\sum_{j=S}^\infty
j^{-\frac{1}{1+\frac{\lambda}{2}}}
\sum_{i=S}^{j-1}
i^{-\frac{1}{1+\frac{\lambda}{2}}}
(j-i)^{-\lambda}\\
&\le
\sum_{j=S}^\infty
j^{-\frac{1}{1+\frac{\lambda}{2}}}
\left(
(j/2)^{-\lambda}	\sum_{i=S}^{[j/2]}
	i^{-\frac{1}{1+\frac{\lambda}{2}}}
+
(j/2)^{-\frac{1}{1+\frac{\lambda}{2}}}
	\sum_{i=[j/2]}^{j-1}
	(j-i)^{-\lambda}
\right)\\
&\le
K\sum_{j=S}^\infty
j^{-\frac{1}{1+\frac{\lambda}{2}}}
\left(
j^{-\lambda + 1-\frac{1}{1+\frac{\lambda}{2}}}
+
j^{-\frac{1}{1+\frac{\lambda}{2}}}(\log j \cdot 1_{\{\lambda\in[1,2)\}}
+
j^{-\lambda+1}1_{\{\lambda\in(0,1)\}})
\right)\\
&\le
K\left(
\sum_{j=S}^\infty
j^{-\frac{2}{1+\frac{\lambda}{2}}}
\log j \cdot 1_{\{\lambda\in[1,2)\}}
+
\sum_{j=S}^\infty
j^{1-\lambda-\frac{2}{1+\frac{\lambda}{2}}}
\cdot 1_{\{\lambda\in(0,1)\}}
\right)\\
&\le
K\left(
S^{1-\frac{2}{1+\frac{\lambda}{2}}}
\log S \cdot 1_{\{\lambda\in[1,2)\}}
+
S^{2-\lambda-\frac{2}{1+\frac{\lambda}{2}}}
\cdot 1_{\{\lambda\in(0,1)\}}
\right).
\end{align*}
Hence, for some positive constant $\rho$, depending only on $\varepsilon, \alpha$ and $\lambda$,
\[
P_2\le K S^{-\rho},
\]
which finishes the proof.
\end{proof}

\begin{lem}
\label{lem:lbound}
Under the conditions of \autoref{thm:main}, for any $\varepsilon\in(0,1)$, there exist positive constants $K$ and $\rho$ depending only on $\varepsilon, \alpha$ and $\lambda$ such that
\begin{align*}
\mathbb P &\left(
\bigcap_{i=0}^{[T-S]}\left\{
	\max_{0\le u \le [y_i^\frac{2}{\alpha}/\theta_i]} X_{r:n}(S+i+u\theta_i y_i^{-\frac{2}{\alpha}})\le y_i - \frac{\theta_i^{\alpha/4}}{y_i}
	\right\}
	\right)
	\\
&\ge
\frac{1}{4}\exp\left(
-(1+\varepsilon)
\int_{S}^T\prob{\sup_{t\in[0,1]}X_{r:n}(t)> f_p(u)}\D u
\right) - K S^{-\rho},
\end{align*}
for any $T-1\ge S\ge K$, where $y_i = f_p(S+i)$ and $\theta_i = y_i^{-\frac{8}{\alpha}}$.
\end{lem}
\begin{proof}
Let, for any $i\ge 0$, $a_i = S + i$ so that $y_i = f_p(a_i)$. Define grid points in the interval $(a_i, a_{i+1}]$ as follows
\begin{equation}
\label{eq:skl}
a_{i,u} = a_i + uq_i,\quad0\le u \le L_i,\quad L_i = [1/q_i], \quad q_i = \theta_i y_i^{-\frac{2}{\alpha}}.
\end{equation}
Finally, put $\hat y_i =  y_i - \theta_i^{\frac{\alpha}{4}}/y_i$.
Similarly as in the proof of \autoref{lem:bound}, using \autoref{lem:Borrel} we have
\begin{align*}
\mathbb P &
\left(
	\bigcap_{i=0}^{[T-S]}
	\left\{
		\max_{0\le u\le L_i} X_{r:n}(a_{i,u})\le \hat y_i
	\right\}
\right)\\
&\ge
\prod_{i=0}^{[T-S]} \prob{\max_{0\le u\le L_i} X_{r:n}(a_{i,u})\le \hat y_i}\\
&\quad -
C_{n,r}
\sum_{0\le i<j\le [T-S]}\sum_{\substack{0\le u\le L_i\\0\le v\le L_j}}
 \left(\hat y_i \hat y_j\right)^{-(n-r)}
 \left(-\tilde A_{a_{i,u}a_{j,v}}^{(r)}\right)^+\exp\left(-\frac{\hat r\left(\hat y_i^2+\hat y_j^2\right)}{2(1+|r(a_{j,v}-a_{i,u})|)}\right)\\
 &=: P_1'-P_2',
\end{align*}
where $\tilde A_{a_{i,u}a_{j,v}}^{(r)}$ is as in \eqref{def:A_il}.

\vb

\noindent\textit{Estimate of $P_1'$.}

\vb

Note that, by \autoref{lem:disc_asymp} combined with \autoref{eq:asymptotics},
\begin{align*}
P_1'
&\ge
 \frac{1}{4}
 \exp\left(
- \sum_{i=0}^{[T-S]}\prob{\max_{0\le u\le L_i} X_{r:n}(a_{i,u})> \hat y_i}
\right)
\ge
 \frac{1}{4}
 \exp\left(
- \sum_{i=0}^{[T-S]}\prob{\sup_{t\in[0,1]}X_{r:n}(t)> \hat y_i}
\right)\\
&\ge
 \frac{1}{4}
 \exp\left(
- (1+\varepsilon)\sum_{i=0}^{[T-S]}\prob{\sup_{t\in[0,1]}X_{r:n}(t)> y_i}
\right)\\
&\ge
\frac{1}{4}\exp\left(
-(1+\varepsilon)
\int_{S}^T\prob{\sup_{t\in[0,1]}X_{r:n}(t)> f_p(u)}\D u
\right),
\end{align*}
provided that $S$ is sufficiently large.

\vb

\noindent\textit{Estimate of $P_2'$.}

\vb

Noting that, for $j\ge i+2$, and any $0\le u\le L_i$, $0\le v\le L_j$;
\[
a_{j,v}-a_{i,u} = a_j + v q_j - a_i- uq_i \ge j-i-1,
\]
we have
\begin{equation}
\label{eq:est1}
\sup_{
	\substack{		
		0\le u\le L_i\\
		0\le v\le L_j
	}
} |r(a_{j,v}-a_{i,u} )|\le
\sup_{|s-s'|\ge j-i-1}
|r(s-s')|=
 r^*(j-i-1)\le r^*(1)<1.
\end{equation}
Since the integrand in definition of $\tilde  A_{a_{i,u}a_{j,v}}^{(r)}$ is continuous and bounded on  $[0,r^*(1)]$, there exists a constant $K$ such that
\[
\left|\tilde  A_{a_{i,u}a_{j,v}}^{(r)}\right|\le K r(a_{j,v}-a_{i,u})\le K r^*(j-i-1)<K.
\]
On the other hand, by
\eqref{eq:rasympt}, there exist positive constants $s_0<1$, such that, for every $0\le s\le s_0$,
\[
\tilde  A_{0s}^{(r)}\ge r(s)\ge 1 - 2|s|^{\alpha}>0.
\]
Hence,
\begin{align}
\label{eq:est2}
(-\tilde A_{a_{i,u}a_{j,v}}^{(r)})^+ = 0, &\quad \text{if}\quad
j=i+1, \quad 1 + v q_j-uq_i \le s_0,\\
\label{eq:est3}
|r(a_{j,v}-a_{i,u})|\le
r^*(s_0)<1,
&\quad\text{if}\quad
j=i+1,\quad 1 + v q_j-uq_i > s_0
\end{align}
Therefore, by \eqref{eq:est1}--\eqref{eq:est3} we obtain
\begin{align*}
P_2' &\le
\sum_{\substack{0\le i\le [T-S]-1 \\ j = i+1}}\sum_{\substack{0\le u\le L_i\\0\le v\le L_j}}
\frac{1}{\sqrt{1- r^*(s_0)}}
\exp\left(-\frac{\hat r(\hat y_i^2+\hat y_j^2)}{2(1+r^*(s_0))}\right)\\
&\quad +
\sum_{\substack{0\le i\le [T-S]-2 \\ i+2 \le j \le [T-S] }}\sum_{\substack{0\le u\le L_i\\0\le v\le L_j}}
\frac{r^*(j-i-1)}{\sqrt{1- r^*(1)}}
\exp\left(-\frac{\hat r(\hat y_i^2+\hat y_j^2)}{2(1+r^*(j-i-1))}\right).
\end{align*}
Completely similar to the estimation of $P_2$ in the proof of \autoref{lem:bound}, we can arrive that there exist positive constants $K$ and $\rho$, independent of $S$ and $T$, such that, for sufficiently large $S$,
\[
P_2'\le K S^{-\rho}.
\]
\end{proof}

The following lemma is a straightforward modification of Lemma 3.1 and 4.1 of \citet{Watanabe70}
and \citet[Lemma 1.4]{Qualls71}.
\begin{lem}
\label{lem:v2}
If \autoref{thm:equiv} is true under the additional condition that for large $t$,
\begin{equation}
\label{eq:condf}
\frac{2}{\hat r}\log t \le f^2(t) \le \frac{3}{\hat r} \log t,
\end{equation}
it is true without the additional condition.
\end{lem}

\section{Proofs of the main results}
\label{sec:Proofs}

\begin{proof}[\bf{Proof of \autoref{thm:equiv}}]
Note that the case $\mathscr I_f <\infty$ is straightforward and does not need any
additional knowledge on process $X_{r:n}$ apart from the assumption of stationarity.
Indeed, for sufficiently large $T$,
\[
\sum_{i = [T]+1}^\infty
\prob{\sup_{t\in [i, i+1]} X_{r:n}(t) > f(i)}
=
\sum_{i= [T]}^\infty \prob{\sup_{t\in[0, 1]} X_{r:n}(t) > f(i+1)}
\le
\mathscr I_f <\infty,
\]
and the Borel--Cantelli lemma completes this part of the proof since $f$ is an increasing function.

Now let $f$ be any increasing function such that $\mathscr I_f\equiv\infty$. With the same notation as in \autoref{lem:bound} with $f$ instead of $f_p$, we find that, for any $S>0$,
\begin{eqnarray*}
\prob{X_{r:n}(s) > f(s)\text{ i.o.}}
&\ge&
\prob{\left\{\sup_{t\in I_i} X_{r:n}(t) > x_i\right\}\quad\text{i.o.}}\\
&\ge&
\prob{\left\{\max_{1\le u\le L_i} X_{r:n}(s_{i,u}) > x_i\right\} \quad\text{i.o.}},
\end{eqnarray*}
where, recall, $s_{i,u}= S + i(1+\varepsilon) + u \theta x_i^{-2/\alpha}$, $L_i=[1/(\theta x_i^{-2/\alpha})]$, $\theta,\varepsilon>0$. Furthermore, for sufficiently large $S$ and $\theta$, cf. estimation of $P_1$,
\begin{equation}
\label{eq:sum}
\sum_{i=0}^\infty
\prob{ \max_{1\le u\le L_i} X_{r:n}(s_{i,u}) > x_i}
\ge
\frac{1-\varepsilon}{1+\varepsilon}
\int_{S}^\infty \prob{\sup_{t\in[0,1]} X_{r:n}(t) > f(u)}\D u=\infty.
\end{equation}
Let $E_i=\{\max_{1\le u\le L_i} X_{r:n}(s_{i,u}) \le x_i\}$, and note that
\[
1-\prob{E_i^c\quad \text{i.o.}}
=
\lim_{m\toi}\prod_{k=m}^\infty\prob{E_k} +
\lim_{m\toi}\left(\prob{\bigcap_{k=m}^\infty E_k} - \prod_{k=m}^\infty\prob{E_k}\right).
\]
The first limit is zero as a consequence of \eqref{eq:sum}, and the second limit will be zero because of the asymptotic independence of the events $E_k$. Indeed, there exist positive constants $K$ and $\rho$, such that for any $n>m$,
\[
A_{m,n}=\left|\prob{\bigcap_{k=m}^n E_k} - \prod_{k=m}^n\prob{E_k}\right|\le K (S+m)^{-\rho},
\]
by the same calculations as in the estimate of $P_2$ in \autoref{lem:bound} after realizing that, by \autoref{lem:v2},
we might restrict ourselves to the case when
\eqref{eq:condf} holds. Therefore, $\prob{E_i^c \text{ i.o.}} =1$, which finishes the proof.
\end{proof}

\noindent \textbf{Proof of \autoref{thm:main}}

\vb

\noindent\textit{Step 1.}
Let $p>1$, then, for every $\varepsilon\in(0,\frac{1}{4})$,
\[
\liminf_{t\toi}\frac{\xi_p(t) -t}{h_p(t)}\ge -(1+2\varepsilon)^2\quad\text{a.s.}
\]
\begin{proof}
Let $\{T_k:k\ge 1\}$ be a sequence such that $T_k\toi$, as $k\toi$.
Put $S_k =T_k -(1+2\varepsilon)^2 h_p(T_k)$. Then by \autoref{lem:bound},
\begin{align*}
\prob{\frac{\xi_p(T_k) - T_k}{h_p(T_k)}\le -(1+2\varepsilon)^2}
&=
\prob{\xi_p(T_k)\le S_k}
=
\prob{\sup_{S_k<t\le T_k}\frac{X_{r:n}(t)}{f_p(t)}< 1}\\
&\le
\exp\left(
-\frac{(1-\varepsilon)}{(1+\varepsilon)}\int_{S_k+1}^{T_k}\prob{\sup_{t\in[0,1]}X_{r:n}(t)> f_p(u)}\D u
\right)
+ 2K T_k^{-\rho},
\end{align*}
where the last inequality follows by the fact that $h_p(t)=o(t)$, so that $S_k\sim T_k$. Note that as $k\toi$
\begin{equation}
\label{eq:logasympt}
\int_{S_k+1}^{T_k}\prob{\sup_{t\in[0,1]}X_{r:n}(t)> f_p(u)}\D u\sim
(1+2\varepsilon)^2 h_p(T_k)\prob{\sup_{t\in[0,1]}X_{r:n}(t)> f_p(T_k)}
=
(1+2\varepsilon)^2 p \log_2 T_k.
\end{equation}
Now take $T_k = \exp(k^{1/p})$. Then
\[
\sum_{k=0}^\infty
\prob{\xi_p(T_k)\le S_k}
\le
2K\sum_{k=0}^\infty k^{-(1+\varepsilon/2)}<\infty.
\]
Hence, by the Borel--Cantelli lemma,
\begin{equation}
\label{eq:step}
\liminf_{k\toi}\frac{\xi_p(T_k)-T_k}{h_p(T_k)}\ge -(1+2\varepsilon)^2\quad\text{a.s.}
\end{equation}
Since $\xi(t)$ is a non-decreasing random function of $t$, for every $T_k\le t\le T_{k+1}$, we have
\begin{align*}
\frac{\xi_p(t)-t}{h_p(t)}\ge 
\frac{\xi_p(T_k)-T_{k}}{h_p(T_k)}- \frac{T_{k+1}-T_k}{h_p(T_k)}.
\end{align*}
For $p >1$ elementary calculus implies
\[
\lim_{k\toi}\frac{T_{k+1}- T_k}{ h_p(T_k)}  =0,
\]
so that
\[
\liminf_{t\toi}\frac{\xi_p(t)-t}{h_p(t)}\ge\liminf_{k\toi}\frac{\xi_p(T_k)-T_k}{h_p(T_k)}\quad\text{a.s.},
\]
which finishes the proof of this step.
\end{proof}

\vb

\noindent\textit{Step 2.}
Let $p>1$, then, for every $\varepsilon\in(0,\frac{1}{4})$,
\[
\liminf_{t\toi}\frac{\xi_p(t)-t}{h_p(t)}\le -(1-\varepsilon)\quad\text{a.s.}
\]
\begin{proof}
As in the proof of the lower bound, put
\[
T_k =\exp(k^{(1+\varepsilon^2)/p}),\quad S_k =T_k -(1-\varepsilon) h_p(T_k),\quad k\ge 1.
\]
Let
\[
B_k = \{\xi_p(T_k)\le S_k\}=\left\{\sup_{S_k<s\le T_k}\frac{X_{r:n}(s)}{f_p(s)}<1\right\}.
\]
It suffices to show $\prob{B_n \text{ i.o.}}=1$, that is
\begin{equation}
\label{eq:ub_goal}
\lim_{m\toi}\prob{\bigcup_{k=m}^\infty B_k} =1.
\end{equation}
Let $a_i^k = S_k + i$ and define grid points in the interval $[a_i^k,a_{i+1}^k]$ as follows
\[
a_{i,u}^k = a_i^k + u q_i^k,\quad
0\le u \le L_i^k,\quad L_i^k = [1/q_i^k],\quad
q_i^k = \theta_i^k (y_i^k)^{-\frac{2}{\alpha}},\quad
\theta_i^k = \left(y_i^k \right)^{-\frac{8}{\alpha}},\quad
y_i^k = f_p(a_i^k).
\]
Put
\[
A_k=\bigcap_{i=0}^{[T_k-S_k]}\left\{\max_{0\le u \le L_i^k} X_{r:n}(a_{i,u}^k)\le y_i^k-\theta_i^k/y_i^k\right\}.
\]
Clearly, for $m\ge 1$,
\[
\prob{\bigcup_{k=m}^\infty A_k} \le \prob{\bigcup_{k=m}^\infty B_k}+\sum_{k=m}^\infty\prob{A_k\cap B_k^c}.
\]
Put $\hat y_i^k = y_i^k-\theta_i^k/y_i^k$. Then,
by \autoref{lem:discreate_approx}, for some constants $K$ independent of $S$ and $T$, which may vary between (and among) lines,
\begin{align*}
\sum_{k=m}^\infty\prob{A_k\cap B_k^c} &\le
\sum_{k=m}^\infty\sum_{i=0}^{[T_k-S_k]}
\prob{
	\max_{0\le u \le L_i^k}X_{r:n}(a_{i,u}^k)\le \hat y_i^k,
	\sup_{s\in [0,1]} X_{r:n}(s)\ge y_i^k
	}\\
&\le 	
K\sum_{k=m}^\infty\sum_{i=0}^{\infty}
(y_i^k)^{\frac{2\hat r}{\alpha}}\Psi(y_i^k) (\theta_i^k)^{\frac{\alpha}{2}-1}\Psi\left(K (\theta^k_i)^{-\frac{\alpha}{4}}\right)\\
&\le
K\sum_{k=m}^\infty\sum_{i=0}^{\infty}
(a_i^k\log^{1-p} a_i^k)^{-1} (\log a_i^k)^{\frac{4}{\alpha}-3\alpha}\exp\left(-\frac{\log^2 a_i^k}{K}\right)\\
&\le
K\sum_{k=m}^\infty\sum_{i=0}^{\infty}(S_k+i)^{-3}
(\log (S_k+i))^{\frac{4}{\alpha}-3\alpha+p-1}\\
&\le
K\sum_{k=m}^\infty S_k^{-1}\le K m^{-4},
\end{align*}
provided $m$ is large enough. Therefore,
\[
\lim_{m\toi} \sum_{k=m}^\infty\prob{A_k\cap B_k^c} = 0
\]
and
\[
\lim_{m\toi}\prob{\bigcup_{k=m}^\infty B_k}\ge\lim_{m\toi} \prob{\bigcup_{k=m}^\infty A_k} .
\]
To finish the proof of \eqref{eq:ub_goal}, we only need to show that
\begin{equation}
\label{eq:generalBC}
\prob{A_n \text{ i.o.}}=1.
\end{equation}
Similarly to \eqref{eq:logasympt}, we have
\[
\int_{S_k}^{T_k}\prob{\sup_{t\in[0,1]}X_{r:n}(t)> f_p(u)}\D u\sim
(1-\varepsilon) p \log_2 T_k.
\]
Now from \autoref{lem:lbound} it follows that
\[
\prob{A_k} \ge
\frac{1}{4}\exp\left(
-(1-\varepsilon^2)
p \log_2 T_k
\right) - K S_k^{-\rho}\ge
\frac{1}{8}k^{-(1-\varepsilon^4)},
\]
for every $k$ sufficiently large. Hence,
\begin{equation}
\label{eq:sumA}
\sum_{k=1}^\infty \prob{A_k} =\infty.
\end{equation}
Applying \autoref{lem:Borrel}, we get for $0\le t<k$
\begin{equation}
\label{eq:AkAt}
\prob{A_kA_t}\le \prob{A_k}\prob{A_t} + M_{k,t},
\end{equation}
where, similarly to the proof of \autoref{lem:bound},
\[
M_{k,t} =
C_{n,r}
\sum_{\substack{0\le i\le [T_k-S_k]\\0\le j\le [T_t-S_t]}}
\sum_{\substack{0\le u\le L_i^k\\0\le v\le L_j^t}}
 \left(\hat y_i^k \hat y_j^t \right)^{-(n-r)}
 \left|\tilde  A_{s_{i,u}^k s_{j,v}^t}^{(r)}\right|
 \exp\left(-\frac{\hat r\left( (\hat y_i^k)^2 + (\hat y_j^t)^2\right)}{2(1+|r(s_{i,u}^k-s_{j,v}^t)|)}\right),
\]
where
\[
\left|\tilde  A_{s_{i,u}^k s_{j,v}^t}^{(r)}\right|\le K  \left| r(s_{i,u}^k - s_{j,v}^t)\right|.
\]
It is easy to see that,
\[
\frac{S_{k+1}-T_k}{T_{k+1}-T_k}\sim 1,\as k,
\]
so that, for $0\le t<k$ and $k$ large enough, and assuming without loss of generality that $\lambda<2$,
\begin{align*}
\left| r(s_{i,u}^k - s_{j,v}^t)\right|&\le r^*(S_k-T_t)\le r^*(S_k-T_{k-1})
\le K r^*\left(\half(T_k-T_{k-1})\right)
\le
2K (T_k-T_{k-1})^{-\lambda}\\&\le \min(1,\lambda)/16.
\end{align*}
Therefore,
\begin{align*}
M_{k,t} &\le
K (T_k-T_{k-1})^{-\lambda}
\sum_{\substack{0\le i\le [T_k-S_k]\\0\le j\le [T_t-S_t]}}
L_i^k L_j^t
 \exp\left(-\frac{\hat r\left( (\hat y_i^k)^2 + (\hat y_j^t)^2\right)}{2(1+\frac{\lambda}{8})}\right) \\
&\le
K (T_k-T_{k-1})^{-\lambda}
L_{[T_k-S_k]}^k L_{[T_t-S_t]}^t
\sum_{\substack{0\le i\le [T_k-S_k]\\0\le j\le [T_t-S_t]}}
(a_i^k)^{-\frac{1}{1+\frac{\lambda}{4}}}
(a_j^t)^{-\frac{1}{1+\frac{\lambda}{4}}}\\
&\le
K (T_k-T_{k-1})^{-\lambda}
\log^{\frac{5}{\alpha}} T_k
\log^{\frac{5}{\alpha}} T_t
\cdot
T_k^{\frac{\lambda}{4}}
T_t^{\frac{\lambda}{4}}\\
&\le
K
T_k^{-\frac{\lambda}{4}}\le K \exp(-\lambda k^{(1+\varepsilon^2)/p}/4).
\end{align*}
Hence we have,
\begin{equation}
\label{eq:sumC}
\sum_{0\le t<k<\infty} M_{k,t} <\infty.
\end{equation}
Now \eqref{eq:generalBC} follows from \eqref{eq:AkAt}, \eqref{eq:sumC} and \eqref{eq:sumA} and the general form of the Borel--Cantelli lemma.
\end{proof}

\vb

\noindent\textit{Step 3.}
If $p\in(0,1]$, then, for every $\varepsilon\in(0,\frac{1}{4})$,
\begin{equation}
\label{eq:step3eq1}
\liminf_{t\toi}\frac{\log\left(\xi_p(t)/t\right)}{h_p(t)/t}\ge -(1+2\varepsilon)^2\quad\text{a.s.}
\end{equation}
and
\begin{equation}
\label{eq:step3eq2}
\liminf_{t\toi}\frac{\log\left(\xi_p(t)/t\right)}{h_p(t)/t}\le -(1-\varepsilon)\quad\text{a.s.}
\end{equation}
\begin{proof}
Put
\[
T_k = \exp(k^{1/p}),\quad S_k = T_k \exp\left(-(1+2\varepsilon)^2h_p(T_k)\right).
\]
Proceeding the same as in the proof of \eqref{eq:step}, one can obtain that
\[
\liminf_{k\toi}\frac{\log\left(\xi_p(T_k)/T_k\right)}{h_p(T_k)/T_k}\ge -(1+2\varepsilon)^2\quad\text{a.s.}
\]
On the other hand it is clear that
\[
\liminf_{t\toi}\frac{\log\left(\xi_p(t)/t\right)}{h_p(t)/t}
\ge
\liminf_{k\toi}\frac{\log\left(\xi_p(T_k)/T_k\right)}{h_p(T_k)/T_k}\quad\text{a.s.}
\]
since
\[
\lim_{k\toi}\frac{\log\left(T_k/T_{k+1}\right)}{h_p(T_k)/T_k}=0.
\]
This proves \eqref{eq:step3eq1}.

Let
\[
T_k = \exp\left(k^{(1+\varepsilon^2)/p}\right),\quad S_k = T_k\exp\left(-(1-\varepsilon)h_p(T_k)\right).
\]
Noting that
\[
\frac{S_{k+1}-T_k}{S_{k+1}}\sim 1\quad\as k,
\]
along the same lines as in the proof of \eqref{eq:ub_goal}, we also have
\[
\liminf_{k\toi}\frac{\log\left(\xi_p(T_k)/T_k\right)}{h_p(T_k)/T_k}\le -(1-\varepsilon)\quad\text{a.s.},
\]
which proves \eqref{eq:step3eq2}.
\end{proof}

{\bf Acknowledgement}:
K. D\polhk{e}bicki was partially supported by
National Science Centre Grant No. 2015/17/B/ST1/01102 (2016-2019).
Research of K. Kosi\'nski was conducted under scientific
Grant No. 2014/12/S/ST1/00491 funded by National Science Centre.

\small\bibliography{biblioteczka}

\end{document}